\documentclass[12pt]{amsart}
\usepackage{amsaddr}
\usepackage{geometry}
\usepackage[mathscr]{euscript}
\usepackage{amssymb}
\usepackage[comma, numbers]{natbib}
\usepackage{enumerate}
\usepackage[english]{babel}
\usepackage [autostyle, english = american]{csquotes}
\MakeOuterQuote{"}

\theoremstyle{plain}
\newtheorem{theorem}{Theorem}[section]

\theoremstyle{plain}
\newtheorem{lemma}[theorem]{Lemma}

\theoremstyle{plain}
\newtheorem{corollary}[theorem]{Corollary}

\theoremstyle{definition}
\newtheorem{definition}[theorem]{Definition}

\theoremstyle{plain}
\newtheorem{proposition}[theorem]{Proposition}

\theoremstyle{remark}

\theoremstyle{definition}
\newtheorem{example}[theorem]{Example}

\theoremstyle{plain}

\theoremstyle{plain}
\newtheorem{conjecture}[theorem]{Conjecture}

\theoremstyle{plain}

\title{Asymptotic dimension of fuzzy metric spaces}
\author{Pawel Grzegrzolka}
\address{Stanford University, 450 Jane Stanford Way, Stanford, CA, 94305, USA}
\email{pawelg@stanford.edu}


\thanks{This research did not receive any specific grant from funding agencies in the public, commercial, or not-for-profit sectors.}


\begin{document}

\begin{abstract}
In this paper, we define asymptotic dimension of fuzzy metric spaces in the sense of George and Veeramani. We prove that asymptotic dimension is an invariant in the coarse category of fuzzy metric spaces. We also show several properties of asymptotic dimension in the fuzzy setting which resemble the properties of asymptotic dimension in the metric setting. We finish with calculating asymptotic dimension of a few fuzzy metric spaces.
\end{abstract}

\maketitle
\tableofcontents

\section{Introduction}
Fuzzy sets were introduced by Zadeh in \cite{Zadeh} to capture the lack of precisely defined criteria of membership. In the usual set theory, given a set $X$, a subset $A$ of $X$, and a point $x_0$ in $X$, we have either $x_0 \in A$ or $x_0 \notin A$. This allows us to describe the set $A$ by its characteristic function $1_A:X \to \{0,1\}$ such that $1_A(x)=1$ if $x \in A$ and $1_A(x)=0$ if $x \notin A$. Zadeh extended the notion of a set by defining a fuzzy set to be a function $f:X \to [0,1]$. The value of $f$ at a particular $x$ is usually understood as the degree of certainty that $x$ belongs to that fuzzy set. Since their introduction, fuzzy sets have been a subject of a remarkably active area of research (for some recent examples, see \cite{fuzzyset5},\cite{fuzzyset1},\cite{fuzzyset2},\cite{fuzzyset3},\cite{fuzzyset4}). 

The notion of fuzziness has been incorporated to the metric setting by Kramosil and Michalek in \cite{michalek}. George and Veeramani slightly modified Kramosil and Michalek's definition of a fuzzy metric space to show that every fuzzy metric space induces a Hausdorff topology (see \cite{main_paper}). The modified version gained much popularity (see for example \cite{popular1},\cite{popular},\cite{Park},\cite{popular2},\cite{open}) and is now considered to be the appropriate notion of metric fuzziness (see \cite{metrizable},\cite{Sapena}). Fuzzy metric spaces account for the uncertainty of measurements of the distance between any two points. In particular, given a parameter $t>0$ and a pair of points $x$ and $y$, a fuzzy metric assigns a "degree of certainty" (a number between $0$ and $1$) such that the distance between $x$ and $y$ is smaller than $t$.

Large-scale geometry (also called coarse topology) has its origins in Mostow's work on rigidity and \v Svarc, Milnor, and Wolf's work on growth of groups (see \cite{nowak}). Large-scale geometry focuses on properties of spaces that are invariant under small perturbations of these spaces (i.e., under quasi-isometries or, more generally, coarse equivalences). One example of such a property is asymptotic dimension (see \cite{asymptotic_dimension},\cite{bedlewo}). The usefulness of asymptotic dimension became apparent due to Yu who proved that the coarse Baum-Connes conjecture holds for proper metric spaces with finite asymptotic dimension (\cite{Yu}). Consequently, he showed that the Novikov Conjecture holds for finitely generated groups whose classifying space has the homotopy type of a finite CW-complex and whose asymptotic dimension is finite when the group is equipped with its word metric (\cite{Yu}). 

Despite its popularity, large-scale geometry has not been studied in the fuzzy metric setting (with the exception of a short note by Zarichnyi; see \cite{rzeszow}). In this paper, we introduce the study of fuzzy metric spaces from the large-scale point of view. In particular, we define the coarse category of fuzzy metric spaces, and we translate many large-scale metric notions to the fuzzy setting. We pay special attention to asymptotic dimension of fuzzy metric spaces. Among other results, we show that asymptotic dimension is an invariant in the coarse category of fuzzy metric spaces, and that asymptotic dimension of fuzzy metric spaces has many properties similar to properties of asymptotic dimension in the metric setting (see for example Theorem \ref{bigtheorem}). We also explicitly calculate asymptotic dimension of several fuzzy metric spaces.

The outline of this paper is as follows. In section \ref{preliminaries}, we review the necessary definitions surrounding fuzzy metric spaces in the sense of George and Veeramani. We also show that an additional condition must be imposed on a  continuous $t$-norm of fuzzy metric spaces to ensure that a finite union of bounded sets is bounded. In section \ref{defofasdim}, we define asymptotic dimension of fuzzy metric spaces, and we prove that asymptotic dimension of the standard fuzzy metric space agrees with asymptotic dimension of the inducing metric space. In section \ref{consequences}, we show several properties of asymptotic dimension of fuzzy metric spaces which are similar to properties of asymptotic dimension in the metric setting. In particular, we prove a slightly weaker fuzzy version of Theorem 19 of \cite{asymptotic_dimension} which lists many equivalent definitions of asymptotic dimension in the metric setting. In section \ref{categorysection}, we establish the coarse category of fuzzy metric spaces, and we prove that the asymptotic dimension of fuzzy metric spaces in an invariant in that category. Finally, in section \ref{examples} we calculate asymptotic dimension of several fuzzy metric spaces. In particular, we show that any non-Archimedean fuzzy metric space has asymptotic dimension $0$.

Throughout the paper, we will assume that $\mathbb{N}=\{n \mid n\geq 1, n \in \mathbb{Z}\}$.

\section{Preliminaries}\label{preliminaries}
In this section, we recall basic definitions from \cite{main_paper} surrounding fuzzy metric spaces. We also recall a few known relationships between metric notions and fuzzy metric notions. Finally, we show that there exist fuzzy metric spaces for which the union of two bounded sets can be unbounded, and we find the condition on the continuous $t$-norm that excludes such examples.

\begin{definition}
A binary operation 
\[*:[0,1] \times [0,1] \to [0,1]\]
is called a \textbf{continuous $t$-norm} if
\begin{enumerate}
\item $*$ is associative and commutative,
\item $*$ is continuous,
\item for all $a\in [0,1],$ $a * 1 = a$,
\item for all $a,b,c,d \in [0,1]$, $a\leq c, b \leq d \Longrightarrow a*b \leq c*d.$
\end{enumerate}
\end{definition}
Notice that conditions $(1), (2),$ and $(3)$ imply that $([0,1],*)$ is a topological monoid. Typical examples of continuous $t$-norms are $a*b=ab$, $a*b=\min\{a,b\}$, and $a*b=\max\{0,a+b-1\}.$ %$a*b=\frac{ab}{a+b-ab}$
\begin{definition}\label{fuzzydef}
 A triple $(X,M,*)$ is a \textbf{fuzzy metric space} if $X$ is a set, $*$ is a continuous $t$-norm, and $M$ is a function
\[M: X \times X \times (0, \infty) \to [0,1]\]
satisfying the following conditions for all $x,y,z \in X$ and all $s,t>0$:
\begin{enumerate}
\item $M(x,y,t)>0$,
\item $M(x,y,t)=1$ if and only if $x=y$,
\item $M(x,y,t)=M(y,x,t)$,
\item $M(x,y,t)*M(y,z,s)\leq M(x,z,t+s)$,
\item $M(x,y,\cdot):(0,\infty) \to [0,1]$ is continuous.
\end{enumerate}
\end{definition}
The adjective "fuzzy" in the above definition comes from the fact that the function $M$ is a fuzzy set. The value of $M(x,y,t)$ is usually understood as the degree of certainty that the distance between $x$ and $y$ is less than $t$. It was shown in \cite{grabiec} that $M(x,y, \cdot)$ is  non-decreasing for all $x,y$ in $X$.

\begin{example}
Let $(X,M,*)$ be a fuzzy metric space. Let $Y\subseteq X$. Then $(Y,M_Y,*)$ is a fuzzy metric space, where $M_Y$ is the restriction of $M$ to $Y \times Y \times (0, \infty)$. We call this fuzzy metric space the \textbf{subspace fuzzy metric space}.
\end{example}

\begin{example}
Let $(X,d)$ be a metric space. Let $a*b=ab$ for any $a,b \in [0,1]$. Define
\[M(x,y,t)=\frac{t}{t+d(x,y)}.\]
Then $(X,M,*)$ is a fuzzy metric space. We call this fuzzy metric space the \textbf{standard fuzzy metric space} and $M$ the \textbf{standard fuzzy metric} induced by the metric $d$. 
\end{example}
It is worth noting that there exist fuzzy metric spaces which are not standard fuzzy metric spaces for any metric (see Example 2.11 in \citep{main_paper}).
\begin{definition}
Let $(X,M,*)$ be a fuzzy metric space. Then the  ball $B(x,r,t)$ with center $x\in X$ and  $r \in (0,1),$ $t>0$ is defined by
\[B(x,r,t)=\{y \in X \mid M(x,y,t)>1-r\}.\]
\end{definition}
It was shown in \cite{main_paper} that the collection of all balls in a fuzzy metric space $(X,M,*)$ induces a first countable and Hausdorff topology on $X$. In fact, this topology is always metrizable (see \cite{metrizable}). When the fuzzy metric space is the standard fuzzy metric space induced by the metric $d$, then the topology induced by the fuzzy metric coincides with the topology induced by the metric $d$ (see \cite{main_paper}).
\begin{definition}
Let $(X,M,*)$ be a fuzzy metric space. A subset $A$ of $X$ is \textbf{bounded} provided there exist $r \in (0,1),$ $t>0$ such that
\[M(x,y,t)>1-r\]
for all $x,y \in A$.
\end{definition}
It was shown in \citep{main_paper} that the bounded sets of a standard fuzzy metric space coincide with the bounded sets of the inducing metric space. However, certain fuzzy metric spaces do not behave as intuitively as one would expect when it comes to boundedness. For example, there are fuzzy metric spaces where the union of two bounded sets is unbounded, even if these two bounded sets share a common point, as the following example shows.

\begin{example}
Let $X=\mathbb{N}$. Equip $X$ with the continuous $t$-norm defined by $a*b=\max\{0, a+b-1\}$. Define $M: X \times X \times (0,\infty) \to [0,1]$ by

\[ M(x,y,t)=\begin{cases} 
      1 & \text{if } x=y \\
     \frac{1}{2} & \text{if } x \neq y, x \neq 1, y \neq 1 \\
      \frac{1}{y} & \text{if } x=1, y \neq 1 \\
      \frac{1}{x} & \text{if } x\neq 1, y = 1 
   \end{cases}
\]
\iffalse
\begin{itemize}
\item $M(x,x,t)=1$ for all $x \in X$ and all $t \in (0,\infty)$,
\item $M(x,y,t)=\frac{1}{2}$ for all $x,y \in X$ such that $x \neq 1$, $y \neq 1$, and $x \neq y$, and for all $t \in (0,\infty)$,
\item $M(1,x,t)=M(x,1,t)=\frac{1}{x}$ for all $x \neq 1$ and all $t \in (0,\infty)$.
\end{itemize}
\fi
It is straightforward to verify that $M$ satisfies all the axioms of a fuzzy metric space. Consider $A=\{1,2\}$ and $B=\{n \mid n \geq 2, n \in X\}$. It is easy to see that both $A$ and $B$ are bounded and share a common point $x=2$. However, $A \cup B$ is unbounded, as $M(1,x,t) \to 0$ when $x \to \infty$.
\end{example}
In the above example, the 4th axiom of a fuzzy metric space was vacuously satisfied, since the left side of that axiom was always zero for distinct $x,y$, and $z$:
\[M(x,y,t)*M(y,z,s)\leq M(x,z,t+s).\]
In other words, the fuzzy equivalent of the triangle inequality did not provide any help in estimating the value of $M(x,z,t+s)$ and consequently allowed us to create the above example. However, as soon as we put an additional constraint on the continuous $t$-norm, such pathological examples are no longer possible, as the following proposition shows.
\begin{proposition}\label{example_proposition}
Let $(X,M,*)$ be a fuzzy metric space such that $a*b \neq 0$ whenever $a\neq 0$ and $b \neq 0$. Then the union of two bounded sets if bounded.
\end{proposition}
\begin{proof}
Let $A,B \subset X$ be bounded. This means that there exist $r,r' \in (0,1)$ and $t,t' >0$ such that $M(x,y,t)>1-r$ for all $x,y \in A$, and $M(x,y,t')>1-r'$ for all $x,y \in B$. Choose $a\in A$ and $b \in B$. Let $x \in A$ and $y \in B$ be arbitrary. Then
\begin{equation*}
\begin{split}
M(x,y, 2t+t') & \geq M(x,a,t)*M(a,b,t)*M(b,y,t')\\
& \geq (1-r)*M(a,b,t)*(1-r').
\end{split}
\end{equation*}
Since by the assumed property of the continuous $t$-norm the right side of the above inequality is a positive number, there exists $s\in (0,1)$ such that
\[M(x,y, 2t+t')>1-s.\]
In fact, we can choose $s\in(0,1)$ large enough so that $1-s<1-r$ and $1-s<1-r'.$ This shows that $M(x,y, 2t+t')>1-s$ for any $x,y \in A \cup B$, proving that  $A\cup B$ is bounded.
\end{proof}

To be able to talk about asymptotic dimension, we must have an effective way of estimating the size and boundedness of sets. Thus, from now on we will always assume that the given continuous $t$-norms have the property that $a*b \neq 0$ whenever $a\neq 0$ and $b \neq 0$. Notice that $a*b=ab$ and $a*b=\min\{a,b\}$ both satisfy this property.

\section{Definition of asymptotic dimension of fuzzy metric spaces}\label{defofasdim}
In this section, we introduce the definition of asymptotic dimension of a fuzzy metric space.  To aid the reader not familiar with the large-scale notions in the metric setting, each definition in the fuzzy metric setting is preceded by the corresponding definition in the metric setting. The definitions in the metric setting come from \cite{asymptotic_dimension}.

The following lemma will  be helpful in translating metric notions to the fuzzy metric setting.
\begin{lemma}\label{useful_lemma}
Let $(X,d)$ be a metric space, and let $(X, M, *)$ be the corresponding standard fuzzy metric space. Then for any $x,y \in X$, $r\in (0,1)$ and $t>0,$
\[M(x,y,t)>1-r \Longleftrightarrow d(x,y)<\frac{rt}{1-r}.\]
\end{lemma}
\begin{proof}
The proof is straightforward and follows from the definition of the standard fuzzy metric space.
\end{proof}

\begin{definition}
A family $\mathcal{U}$ of subsets of a metric space $(X,d)$ is \textbf{uniformly bounded} (in $(X,d)$) if there exists $t>0$ such that $d(x,y)<t$ for all $x,y \in U$ for any $U \in \mathcal{U}$.
\end{definition}

\begin{definition}\label{uniformly_bounded}
A family $\mathcal{U}$ of subsets of a fuzzy metric space $(X,M, *)$ is \textbf{uniformly bounded} (in $(X,M,*)$) if there exist $r\in (0,1)$ and $t>0$ such that $M(x,y,t)>1-r$ for all $x,y \in U$ for any $U \in \mathcal{U}$.
\end{definition}

\begin{proposition} \label{first_prop}
Let $(X,d)$ be a metric space, and let $(X, M, *)$ be the corresponding standard fuzzy metric space. Let $\mathcal{U}$ be a family of subsets of $X$. Then $\mathcal{U}$ is uniformly bounded in $(X,d)$ if and only if it is uniformly bounded in $(X, M, *)$.
\end{proposition}
\begin{proof}
Lemma \ref{useful_lemma} shows the backward direction of the proposition.
Letting $r=\frac{1}{2}$, the equivalence in Lemma \ref{useful_lemma} turns into 
\[M(x,y,t)>\frac{1}{2} \Longleftrightarrow d(x,y)<t.\]
This proves the forward direction of the proposition.
\end{proof}

\begin{definition}
Two subsets $U$ and $U'$ of a metric space $(X,d)$ are called \textbf{$t$-disjoint} for some $t>0$ if $d(U,U')>t$, where
\[d(U,U')=\inf\{d(x,y) \mid x\in U, y \in U'\}.\]
A family $\mathcal{U}$ of subsets of $(X,d)$ is called \textbf{$t$-disjoint} if any two distinct elements of $\mathcal{U}$ are $t$-disjoint.
\end{definition}

\begin{definition}\label{rtdisjoint}
Two subsets $U$ and $U'$ of a fuzzy metric space $(X,M,*)$ are called \textbf{$(r,t)$-disjoint} for some $r\in(0,1)$ and $t>0$ if $M(U,U',t)<1-r$, where
\[M(U,U',t)=\sup\{M(x,y,t) \mid x\in U, y \in U'\}.\]
A family $\mathcal{U}$ of subsets of $(X,M,*)$ is called \textbf{$(r,t)$-disjoint} if any two distinct elements of $\mathcal{U}$ are $(r,t)$-disjoint.
\end{definition}
It is easy to see (by using Lemma \ref{useful_lemma}) that when $(X,d)$ is a metric space, $(X, M,*)$ is the corresponding standard fuzzy metric space, and $U,U' \subseteq X,$ then $U$ and $U'$ are $t$-disjoint for some $t>0$ if and only if they are $(r',t')$-disjoint for some $r' \in (0,1)$ and some $t'>0$. 
\iffalse
\begin{proposition}
Let $(X,d)$ be a metric space, and let $(X, M, *)$ be the corresponding standard fuzzy metric space. Let $U$ and $U'$ be subsets of $X$. Then $U$ and $U'$ are $t$-disjoint for some $t>0$ if and only if they are $(r',t')$-disjoint for some $r' \in (0,1)$ and some $t'>0$.
\end{proposition}
\begin{proof}
To see the backward direction, let $U$ and $U'$ be $(r',t')$  disjoint for some $r'\in(0,1), t'>0$. Then $M(U,U',t') <1-r'.$ In particular, $M(x,y,t')<1-r'$ for any $x \in U$ and $y \in U'$. Then the equivalence
\[M(x,y,t')< 1-r' \Longleftrightarrow  d(x,y)> \frac{r't'}{1-r'}\]
shows that $d(U,U') \geq \frac{r't'}{1-r'}$, since $x$ and $y$ are arbitrary elements of $U$ and $U'$, respectively. Setting $t=\frac{r't'}{2(1-r')}$ gives
\[d(U,U') \geq \frac{r't'}{1-r'} >t,\]
showing that $U$ and $U'$ are $t$-disjoint.

Conversely, let $U$ and $U'$ be $t$-disjoint for some $t>0$. Then $d(U,U')>t$. In particular, for any $x \in U$ and $y \in U'$, $d(x,y)>t$. Setting $r=\frac{1}{2}$ gives
\[d(x,y)> t \Longleftrightarrow  M(x,y,t)< \frac{1}{2}.\]
Since $x$ and $y$ were arbitrary elements of $U$ and $U'$, respectively, this shows that 
\[M(U,U',t) \leq \frac{1}{2}.\]
Setting $r'=\frac{1}{4}$ and $t'=t$ gives
\[M(U,U',t') \leq \frac{1}{2} <1-r'\]
showing that $U$ and $U'$ are $(r',t')$-disjoint.
\end{proof}
\fi

\begin{definition}
Let $(X,d)$ be a metric space. We say that \textbf{asymptotic dimension} of $X$ does not exceed $n$ and write $\operatorname{asdim}X \leq n$ provided for every $r>0$, there exist $r$-disjoint families $\mathcal{U}^0, \dots, \mathcal{U}^n$ of subsets of $X$ such that $\bigcup_i \mathcal{U}^i$ is a uniformly bounded cover of $(X,d).$ If no such $n$ exists, we say that $\operatorname{asdim} X=\infty$. If $\operatorname{asdim} X \leq n$ but $asdim X\ \not\leq n-1,$ then we say that $\operatorname{asdim}X=n$.
\end{definition}

\begin{definition}
Let $(X,M,*)$ be a fuzzy metric space. We say that \textbf{asymptotic dimension} of $X$ does not exceed $n$ and write $\operatorname{asdim} X \leq n$ provided for every $r\in (0,1)$ and $t>0$, there exist $(r,t)$-disjoint families $\mathcal{U}^0, \dots, \mathcal{U}^n$ of subsets of $X$ such that $\bigcup_i \mathcal{U}^i$ is a uniformly bounded cover of $(X,M,*).$ If no such $n$ exists, we say that $\operatorname{asdim} X=\infty$. If $\operatorname{asdim} X \leq n$ but $\operatorname{asdim} X\ \not\leq n-1,$ then we say that $\operatorname{asdim}X=n$.
\end{definition}

\iffalse
\begin{definition}
Let $(X,d)$ be a metric space. We say that the \textbf{asymptotic dimension} of $X$ does not exceed $n$ and write $\operatorname{asdim} X \leq n$ provided for every uniformly bounded open cover $\mathcal{V}$ of $X$ there is a uniformly bounded cover $\mathcal{U}$ of $X$ of multiplicity $\leq n+1$ so that $\mathcal{V}$ refines $\mathcal{U}$. We write $\operatorname{asdim} X =n $ if it is true that $\operatorname{asdim} X \leq n$ and $\operatorname{asdim} X \not\leq n-1$.
\end{definition}
\begin{definition}
Let $(X,M,*)$ be a fuzzy metric space. We say that the \textbf{asymptotic dimension} of $X$ does not exceed $n$ and write $\operatorname{asdim} X \leq n$ provided for every uniformly bounded open cover $\mathcal{V}$ of $X$ there is a uniformly bounded cover $\mathcal{U}$ of $X$ of multiplicity $\leq n+1$ so that $\mathcal{V}$ refines $\mathcal{U}$. We write $\operatorname{asdim} X =n $ if it is true that $\operatorname{asdim} X \leq n$ and $\operatorname{asdim} X \not\leq n-1$.
\end{definition}
\fi

When we need to distinguish between  asymptotic dimension of a metric space and asymptotic dimension of a fuzzy metric space, we will write $\operatorname{asdim}(X,d)$ and $\operatorname{asdim}(X,M,*)$, respectively.

\begin{proposition}\label{standardandfuzzy}
Let $(X,d)$ be a metric space, and let $(X, M, *)$ be the corresponding standard fuzzy metric space. Then 
%\[\operatorname{asdim}(X,d)= n \Longleftrightarrow \operatorname{asdim}(X,M,*)= n.\]
\[\operatorname{asdim}(X,d)=\operatorname{asdim}(X,M,*).\]
\end{proposition}
\begin{proof}
Assume that $\operatorname{asdim}(X,d)\leq n$. Let $r\in (0,1)$ and $t>0$ be arbitrary. Let $s\in (0,1)$ be such that $s>r$. Since $\operatorname{asdim}(X,d)\leq n$, there exist $\frac{st}{1-s}$-disjoint families $\mathcal{U}^0, \dots, \mathcal{U}^n$ of subsets of $X$ such that $\bigcup_i \mathcal{U}^i$ is a uniformly bounded cover of $(X,d)$. By Proposition \ref{first_prop}, these families form a uniformly bounded cover of $(X,M,*)$. By Lemma \ref{useful_lemma}, we have 
\[M(x,y,t)<1-s \Longleftrightarrow d(x,y)>\frac{st}{1-s}\]
for any $x,y \in X$. This shows that 
\[M(x,y,t)<1-s<1-r\]
for any $x \in U \in  \mathcal{U}^i$ and $y \in U' \in  \mathcal{U}^i$ such that $U \neq U'$, where $i$ is any integer such that $0 \leq i \leq n$. Thus, 
\[M(U,U',t)<1-r\]
for any $U, U' \in  \mathcal{U}_i$ such that $U \neq U'$, where $i$ is any integer such that $0 \leq i \leq n$. This proves that the families $\mathcal{U}^0, \dots, \mathcal{U}^n$ are $(r,t)$-disjoint, showing that $\operatorname{asdim}(X,M,*)\leq n.$

Conversely, assume that $\operatorname{asdim}(X,M,*)\leq n.$ Let $t>0$ be arbitrary. Since $\operatorname{asdim}(X,M,*)\leq n,$ there exist $(\frac{1}{4},t)$-disjoint families $\mathcal{U}^0, \dots, \mathcal{U}^n$ of subsets of $X$ such that $\bigcup_i \mathcal{U}^i$ is a uniformly bounded cover of $(X,M,*)$. By Proposition \ref{first_prop}, these families form a uniformly bounded cover of $(X,d)$. By Lemma \ref{useful_lemma} (with $r=\frac{1}{4}$), we have 
\[M(x,y,t)<1-\frac{1}{4} \Longleftrightarrow d(x,y)>\frac{\frac{1}{4}t}{1-\frac{1}{4}}=\frac{1}{3}t\]
for any $x,y \in X$. This shows that 
\[d(x,y)>\frac{1}{3}t>t\]
for any $x \in U \in  \mathcal{U}^i$ and $y \in U' \in  \mathcal{U}^i$ such that $U \neq U'$, where $i$ is any integer such that $0 \leq i \leq n$. Thus, 
\[d(U,U')>t\]
for any $U, U' \in  \mathcal{U}^i$ such that $U \neq U'$, where $i$ is any integer such that $0 \leq i \leq n$. This proves that the families $\mathcal{U}^0, \dots, \mathcal{U}^n$ are $t$-disjoint, showing that $\operatorname{asdim}(X,d)\leq n.$

The above proof also implies that if one of the asymptotic dimensions is infinite, then so is the other.
\end{proof}

\iffalse
\begin{corollary}
Let $(X,d)$ be a metric space, and let $(X, M, *)$ be the corresponding standard fuzzy metric space. Then 
\[\operatorname{asdim}(X,d)= \infty \Longleftrightarrow \operatorname{asdim}(X,M,*)= \infty. \eqno\qed\]
\end{corollary}
\fi

\begin{proposition}\label{subspace}
Let $(X,M,*)$ be a fuzzy metric space. Let $Y \subseteq X$. Let $(Y, M_Y, *)$ be a subspace fuzzy metric space. Then $\operatorname{asdim}Y \leq \operatorname{asdim}X$. 
\end{proposition}
\begin{proof}
The proof is straightforward and closely imitates the proof in the metric setting (see Proposition 23 in \cite{asymptotic_dimension}).
%The proposition is clearly true if $\operatorname{asdim}X= \infty$, so assume that $\operatorname{asdim}X=n$. Let $r\in(0,1)$ and $t>0$ be given. Since $\operatorname{asdim}X=n$, there exist $(r,t)$-disjoint families $\mathcal{U}^0, \dots, \mathcal{U}^n$ in $X$ such that $\bigcup_i \mathcal{U}^i$ is a uniformly bounded cover of $X$. Clearly the restriction of the families $\mathcal{U}^0, \dots, \mathcal{U}^n$ to $Y$ is a uniformly bounded cover of $Y$. Since these families are also $(r,t)$-disjoint, this shows that $\operatorname{asdim}Y\leq n$.
\end{proof}

For examples of explicit calculations of asymptotic dimension of certain fuzzy metric spaces, the reader is referred to Section \ref{examples}.

\section{Properties of asymptotic dimension of fuzzy metric spaces}\label{consequences}

To state the properties of asymptotic dimension similar to Theorem 19 in \cite{asymptotic_dimension}, we need to translate the notions of $t$-multiplicity, $t$-neighborhood, and Lebesgue number to the fuzzy setting. For the definitions in the metric setting, the reader is referred to \cite{asymptotic_dimension}. %To aid the geometric intuition and show the motivation behind the definitions in the fuzzy metric setting, each definition in the fuzzy metric setting is preceded by the definition in the metric setting.  The definitions in the metric setting come from \cite{asymptotic_dimension}.
\iffalse
\begin{definition}
Let $t>0$. The \textbf{$t$-multiplicity} of a cover $\mathcal{U}$ in a metric space $(X,d)$ is the largest integer $n$ so that there exists $x\in X$ so that the ball $B(x,t)$ meets $n$ of the sets in $\mathcal{U}$.
\end{definition}
\fi

\begin{definition}\label{multiplicity}
Let $r \in (0,1)$ and $t>0$. The \textbf{$(r,t)$-multiplicity} of a cover $\mathcal{U}$ in a fuzzy metric space $(X,M,*)$ is the largest integer $n$ so that there exists $x\in X$ so that $B(x,r,t)$ meets $n$ of the sets in $\mathcal{U}$.
\end{definition}
\begin{definition}\label{nbhddef}
Let $r\in (0,1)$ and $t>0$. The \textbf{$(r,t)$-neighborhood} of a subset $U$ of a fuzzy metric space $(X,M,*)$ is defined by
\[N_{r,t}(U)=\{x \in X \mid \exists y \in U \text{ such that } M(x,y,t)>1-r\}.\]
\end{definition}

\begin{definition}\label{lpair}
A \textbf{Lebesgue pair} of a cover $\mathcal{U}$ of a fuzzy metric space $(X,M,*)$ is a pair of numbers $r\in (0,1)$ and $t>0$ so that for any $x \in X$, $B(x,r,t) \subseteq U$ for some $U \in \mathcal{U}.$
\end{definition}

One of the first mentions of the Lebesgue property for fuzzy metric spaces can be found in \cite{uniform}. By using Lemma \ref{useful_lemma}, one can show the equivalence of $(r,t)$-multiplicity, $(r,t)$-neighborhood, and Lebesgue pair to the corresponding definitions in the metric setting. The equivalence is similar to the equivalence of $t$-disjointness and $(r,t)$-disjointness mentioned after Definition \ref{rtdisjoint}.

We will use the following proposition about $(r,t)$-neighborhoods in Theorem \ref{bigtheorem}.

\begin{proposition}\label{nbhd}
If $\mathcal{U}$ is a family of uniformly bounded subsets of a fuzzy metric space $(X,M,*)$, then $\{N_{r,t}(U) \mid U \in \mathcal{U}\}$ is also a uniformly bounded family of subsets of $X$ for any $r\in (0,1)$, $t>0$. If $\mathcal{U}$ is a cover of $X$, then so is $\{N_{r,t}(U) \mid U \in \mathcal{U}\}$.
\end{proposition}
\begin{proof}
The proof is straightforward.
\iffalse
Let $r \in (0,1)$ and $t>0$ be arbitrary. Since $\mathcal{U}$ is uniformly bounded, there exist $r'\in (0,1)$ and $t'>0$ such that $M(x,y,t')>1-r'$ for any $x,y \in U$ for any $U \in \mathcal{U}$. Let $U \in \mathcal{U}$ be arbitrary. Let $x,y \in N_{r,t}(U)$. Then there exist $x' \in U$ and $y' \in U$ such that $M(x,x',t)>1-r$ and $M(y,y',t)>1-r$. Since $x',y' \in U$, we also know that $M(x',y',t')>1-r'$. Consequently,
\begin{equation*}
\begin{split}
M(x,y,2t+t') & \geq  M(x,x',t)*M(x',y',t')*M(y',y,t)\\
& \geq (1-r)*(1-r')*(1-r).
\end{split}
\end{equation*}
Since the right side of the above inequality is positive, this shows that the family $\{N_{r,t}(U) \mid U \in \mathcal{U}\}$ is also uniformly bounded. Since $U \subseteq N_{r,t}(U)$ for any $U\in\mathcal{U}$, if $\mathcal{U}$ is a cover of $X$, then so is $\{N_{r,t}(U) \mid U \in \mathcal{U}\}$.
\fi
\end{proof}

Recall that a cover $\mathcal{V}$ \textbf{refines} a cover $\mathcal{U}$ if every element of $\mathcal{V}$ is in some element of $\mathcal{U}$. The \textbf{multiplicity} of a cover is the largest number of elements of the cover meeting any point of the space.

\begin{theorem}\label{bigtheorem}
Let $(X,M,*)$ be a fuzzy metric space. Then $\operatorname{asdim} X \leq n$ implies all of the following conditions:
\begin{enumerate}
\item for every $r\in (0,1)$ and $t>0$, there exists a uniformly bounded cover $\mathcal{U}$ of $X$ with $(r,t)$-multiplicity $\leq n+1$, \label{condition1}
\item for every $r\in (0,1)$ and $t>0$, there exists a uniformly bounded cover $\mathcal{V}$ of $X$ with a Lebesgue pair $r,t$ and multiplicity $\leq n+1,$ \label{condition2}
\item for every uniformly bounded cover $\mathcal{U}$ of $X$, there is a uniformly bounded cover $\mathcal{V}$ of $X$ of multiplicity $\leq n+1$ so that $\mathcal{U}$ refines $\mathcal{V}.$\label{condition3}
\end{enumerate}
\end{theorem}
\begin{proof}
We will show $(\operatorname{asdim}X\leq n) \implies (1) \implies (2) \implies (3)$.

$(\operatorname{asdim} \leq n) \implies (1)$. Let $r \in (0,1)$ and $t>0$ be arbitrary. Let $t'=2t$ and choose $r'\in (0,1)$ so that $1-r'<(1-r)*(1-r)$. By assumption, there exist uniformly bounded $(r',t')$-disjoint families $\mathcal{U}^0,\dots, \mathcal{U}^n$ of subsets of $X$ such that $\bigcup_i \mathcal{U}^i$ is a cover of $X.$ Let $\mathcal{U}=\bigcup_i \mathcal{U}^i$. Clearly $\mathcal{U}$ is uniformly bounded. Let $x \in X$ be arbitrary, and consider $B(x,r,t)$. If $B(x,r,t)$ intersects two elements $U$ and $U'$ of $\mathcal{V}$, then there exist $u \in U$ and $u' \in U'$ such that $M(x,u,t)>1-r$ and $M(x,u',t)>1-r$. Thus,
\begin{equation*}
\begin{split}
M(u,u',t') & = M(u,u',t+t)\\
& \geq M(u,x,t)* M(x,u',t) \\
& \geq (1-r)*(1-r) \\
& > 1-r'. \\
\end{split}
\end{equation*}
This implies that $U$ and $U'$ are not $(r',t')$-disjoint and must have come from different families $\mathcal{U}^i$ and $\mathcal{U}^j$ with $i\neq j$. Since there are $n+1$ distinct families, there can be at most $n+1$ elements of $\mathcal{V}=\bigcup_i \mathcal{U}^i$ that intersect $B(x,r,t)$. Consequently, $\mathcal{V}$ has $(r,t)$-multiplicity $\leq n+1.$

$(1) \implies (2)$. Let $r\in (0,1)$ and $t>0$ be arbitrary. Let $t'=2t$ and choose $r'\in (0,1)$ so that $1-r'<(1-r)*(1-r)$. By assumption, there exists a uniformly bounded cover $\mathcal{U}$ of $X$ with $(r',t')$-multiplicity $\leq n+1.$ Define $\overline{U}=N_{r',t'}(U)$, and let $\mathcal{V}=\{\overline{U} \mid U \in \mathcal{U}\}$. By Proposition \ref{nbhd}, $\mathcal{V}$ is uniformly bounded. Notice that $\mathcal{V}$ has a Lebesgue pair $r,t$ because for any $x \in X$, $B(x,r,t)$ intersects some $U \in \mathcal{U}$, which implies that $B(x,r,t) \subseteq \overline{U}$. To see it, let $y \in B(x,r,t)$. Since $B(x,r,t)$ intersects $U$, there exists $u\in U$ such that $M(x,u,t)>1-r$. Consequently,
\begin{equation*}
\begin{split}
M(u,y,t') & = M(u,y,t+t)\\
& \geq M(u,x,t)* M(x,y,t) \\
& \geq (1-r)*(1-r) \\
& > 1-r'. \\
\end{split}
\end{equation*}
Thus, $y$ belongs to the $(r',t')$-neighborhood of $U$, showing that $B(x,r,t) \subseteq \overline{U}$.
%Consequently, $B(x,r,t) \subseteq \overline{U}$, and $\mathcal{W}$ has a Lebesgue pair $r,t$. 
To see that the multiplicity of $\mathcal{V}$ is $\leq n+1,$ assume for contradiction that there exists $x\in X$ that intersects $n+2$ elements of $\mathcal{V}$. Call these elements $V_1, \dots , V_{n+2},$ and the corresponding elements of $\mathcal{U}$ by $U_1,\dots, U_{n+2}$. This means that there exist $u_i \in U_i$ for all $1 \leq i \leq n+2$ such that $M(x,u_i,t')>1-r'$. Thus, the ball $B(x, r',t')$ intersects every $U_i$, which contradicts the $(r',t')$-multiplicity of the cover $\mathcal{U}$.

$(2) \implies (3)$. Let $\mathcal{U}$ be a uniformly bounded cover of $X$. This means that there exist $r\in (0,1)$ and $t>0$ such that $M(x,y,t)>1-r$ for all $x,y \in U$ for any $U \in \mathcal{U}$. By assumption, there exists a uniformly bounded cover $\mathcal{V}$ of $X$ with a Lebesgue pair $r,t$ and multiplicity $\leq n+1$. It is left to show that $\mathcal{U}$ refines $\mathcal{V}$. Let  $U \in \mathcal{U}$. Let $u \in U$. Then $U \subseteq B(u,r,t).$ Since $\mathcal{V}$ is a cover with a Lebesgue pair $r,t$, there exists $V \in \mathcal{V}$ such that $U \subseteq B(u,r,t) \subseteq V,$ showing that $U$ refines $V$.
\iffalse
$(1) \implies (3)$. Let $r\in (0,1)$ and $t>0$. Let $\mathcal{U}=\{B(x,r,t) \mid x\in X\}$. It is easy to see that $\mathcal{U}$ is a uniformly bounded cover. Since $\operatorname{asdim}(X)\leq n+1$, there exists a uniformly bounded cover $\mathcal{W}$ of multiplicity $\leq n+1$ such that $\mathcal{V}$ refines $\mathcal{W}$. Notice that any ball $B(x,r,t)$ is an element of $\mathcal{V},$ and thus there exists a $W \in \mathcal{W}$ containing it. Thus, the uniformly bounded cover $\mathcal{W}$ has a Lebesgue pair $r,t.$
\fi
\end{proof}

The reader is encouraged to compare the above theorem with Theorem 19 in \cite{asymptotic_dimension}.

\begin{corollary}
For $n=0$, conditions (\ref{condition1}), (\ref{condition2}), and (\ref{condition3}) in Theorem \ref{bigtheorem} are equivalent to the statement that $\operatorname{asdim}X =0.$
\end{corollary}
\begin{proof}
By Theorem \ref{bigtheorem}, it is enough to show that condition (\ref{condition3}) implies $\operatorname{asdim}X =0.$ Let $r\in(0,1)$ and $t>0$ be arbitrary. Let $r' \in (0,1)$ be such that $1-r'<1-r$. Consider the uniformly bounded cover $\mathcal{U}=\{B(x,r',t) \mid x \in X\}$. By assumption, there exists a uniformly bounded cover $\mathcal{V}$ of multiplicity $1$ such that $\mathcal{U}$ refines $\mathcal{V}$. This implies that for any $v \in V$, $B(x,r',t)\subseteq V$. This means that for any $v' \in V'$ such that $V \neq V'$, we must have
\[M(v,v',t)\leq 1-r' <1 -r,\]
showing that $\mathcal{V}$ is $(r,t)$-disjoint.
\iffalse
By condition (\ref{condition2}), there exists a uniformly bounded cover $\mathcal{U}$ of $X$ with $(r',t)$-multiplicity $\leq 1$. It is then enough to show that $\mathcal{U}$ is $(r,t)$-disjoint. Let $U,U' \in \mathcal{U}$ be arbitrary and distinct. By assumption, for any $u \in U$, we have that $B(u,r',t)$ intersects only $U$. This means that for any $u' \in U'$, it must be that 
\[M(u,u',t)\leq 1-r' <1 -r,\]
showing that $\mathcal{U}$ is $(r,t)$-disjoint.
\fi
\end{proof}

\begin{conjecture}
Conditions (\ref{condition1}), (\ref{condition2}), and (\ref{condition3}) in Theorem \ref{bigtheorem} are equivalent to the statement that $\operatorname{asdim}X \leq n.$ Consequently, they can be used as alternative definitions of asymptotic dimension of fuzzy metric spaces.
\end{conjecture}

\section{Coarse category of fuzzy metric spaces}\label{categorysection}
In this section, we define the coarse category of fuzzy metric spaces. We also show that asymptotic dimension is an invariant in this category.

The reader is encouraged to compare the definitions in this section with the definitions in the metric setting in \cite{permanence}.

\begin{definition}
Let $(X,M_1,*_1)$ and $(Y, M_2,*_2)$ be fuzzy metric spaces. Let $f:X \to Y$ be a map. Then $f$ is called \textbf{uniformly expansive} if for all $A>0$ and $t>0$, there exist $B>0$ and $t'>0$ such that
\[ M_1(x,y,t) \geq A \implies M_2(f(x),f(y),t') \geq B\]
for any $x,y \in X$.
The function $f$ is called \textbf{effectively proper} if for all $C>0$ and $t>0$, there exist $D>0$ and $t'>0$ such that
\[M_2(f(x),f(y),t) \geq C \implies M_1(x,y,t') \geq D\]
for any $x,y \in X$.
The function $f$ is called a \textbf{coarse embedding} if it is both uniformly expansive and effectively proper. The function $f$ is called a \textbf{coarse equivalence} if it is a coarse embedding which is \textbf{coarsely onto} in the sense that there exists $r \in (0,1)$ and $t>0$ such that the $N_{r,t}$-neighborhood of $f(X)$ is $Y$, i.e.,
\[\forall \, y \in Y, \exists \, x \in X \text{ such that } M_2(f(x),y,t)>1-r.\] 
\end{definition}

\begin{definition}
Let $X$ be a set and let $(Y, M,*)$ be a fuzzy metric space. Let $f:X \to Y$ and $g:X\to Y$ be functions. Then $f$ is \textbf{close} to $g$, denoted $f \sim g$, if there exists $r \in (0,1)$ and $t>0$ such that 
\[M(f(x),g(x),t)>1-r\]
for all $x \in X$.
\end{definition}
It is easy to verify that the above defined closeness is an equivalence relation on the set of functions from $X$ to $(Y,M,*)$. We will denote the equivalence class of the function $f$ under the closeness relation by $[f]$. If $X$ is a fuzzy metric space, it is easy to see that if $f$ is close to $g$ and $f$ is uniformly expansive, then $g$ is also uniformly expansive.

\begin{theorem}
The collection of fuzzy metric spaces together with the closeness classes of uniformly expansive maps between the fuzzy metric spaces forms a category. We will call this category the \textbf{coarse category of fuzzy metric spaces}.
\end{theorem}
\begin{proof}
We only need to verify that the composition of morphisms $[g] \circ [f] = [g \circ f]$ is well-defined, as then the associativity and the existence of the identity morphism for each object are immediate. Let $(X,M_1,*_1)$, $(Y,M_2,*_2)$ and $(Z,M_3,*_3)$ be fuzzy metric spaces. Let $f,f':X\to Y$ be functions such that $f \sim f'$ and $g,g':Y \to Z$ be functions such that $g \sim g'$. We need to show that $g \circ f \sim g' \circ f'$. Since $f \sim f'$, there exists $r \in (0,1)$ and $t>0$ such that $M_2(f(x),f'(x),t)>1-r$ for any $x \in X$. Let $x \in X$ be arbitrary. Then $M_2(f(x),f'(x),t)>1-r$. Since $g\sim g'$, there exists $r'\in(0,1)$ and $t'>0$ such that $M_3(g(y),g'(y),t')>1-r'$ for any $y \in Y$. In particular, $M_3(g(f(x)),g'(f(x)),t')>1-r'$. Since $g'$ is uniformly expansive, there exist $B>0$ and $t''>0$ such that $M_2(y_1,y_2,t) \geq 1-r$ implies $M_3(g'(y_1),g'(y_2),t'') \geq B$ for any $y_1, y_2 \in Y$. In particular $M_3(g'(f(x)),g'(f'(x)),t'') \geq B$, since  $M_2(f(x),f'(x),t)>1-r$. Without loss of generality we can assume that $B<1$. Putting it all together, we get
\begin{equation*}
\begin{split}
M_3(g(f(x)), g'(f'(x)),t'+t'') & \geq M_3(g(f(x)), g'(f(x)), t') *_3 M_3(g'(f(x)), g'(f'(x)),t'') \\
& \geq (1-r') *_3 B
\end{split}
\end{equation*}
Since the right side of this inequality is positive, there exists $s \in (0,1)$ such that 
\[M_3(g(f(x)), g'(f'(x)),t'+t'')>1-s.\]
Since $t',t'',$ and $s$ did not depend on the choice of $x$ in $X$, this shows that $g \circ f \sim g' \circ f'$.
\end{proof}

The following proposition shows that coarse equivalences are the isomorphisms in the coarse category of fuzzy metric spaces. The reader is encouraged to compare the following proposition with Lemma 2.2 in \cite{permanence}.

\begin{proposition}\label{different_way}
Let $(X,M_1,*_1)$ and $(Y, M_2,*_2)$ be fuzzy metric spaces. Let $f:X \to Y$ be a function. Then $f$ is a coarse equivalence if and only if $f$ is uniformly expansive and there exists a uniformly expansive $g: Y\to X$ such that the compositions $f \circ g$ and $g \circ f$ are close to the identity of $Y$ and $X$, respectively. The function $g$ is called a \textbf{coarse inverse} of $f$.
\end{proposition}
\begin{proof}
Let $f$ be a coarse equivalence. Since $f$ is coarsely onto, there exist $r\in(0,1)$ and $t>0$ such that for all $y \in Y$, there exists $x \in X$ such that $M_2(f(x),y,t)>1-r$. Consequently, we can define $g:Y \to X$ by
\[g(y)= \text{some $x$ satisfying $M_2(f(x),y,t)>1-r$}.\]
To see that $g$ is uniformly expansive, Let $A>0$, $t'>0$ be arbitrary, and assume that $M_2(y_1,y_2,t') \geq A$ for some $y_1,y_2 \in Y$. Let $g(y_1)=x_1$ and $g(y_2)=x_2$. Then we have
\begin{equation*}
\begin{split}
M_2(f(x_1), f(x_2),t'+2t) & \geq M_2(f(x_1),y_1,t)*_2 M_2(y_1,y_2,t')*_2M_2(y_2,f(x_2),t)\\
& \geq (1-r) *_2 A *_2(1-r)
\end{split}
\end{equation*}
Since the right side of this inequality is positive, the effective properness of $f$ implies that there exist $D>0$ and $t''>0$ such that $M_1(x_1,x_2,t'')\geq D$, which is the same as $M_1(g(y_1),g(y_2),t'')\geq D$. This shows the uniform expansiveness of $g$. To see that $f \circ g$ is close to the identity of $Y$, let $y \in Y$. Notice that by the definition of $g$, $M_2(f(g(y)),y,t)>1-r$, i.e., $f \circ g$ is close to the identity of $Y$. To show that $g \circ f$ is close to the identity of $X$, let $x \in X$. Notice that by the definition of $g$, we have that $M_2(f(x), f(g(f(x))), t)>1-r$. The effective properness of $f$ implies then that there exists $D'>0$ and $t'''>0$ such that $M_1(x, f(g(x)),t''')\geq D'$. This means that there exists $s \in (0,1)$ such that $M_1(x, f(g(x)),t''')> 1-s$, showing that $f \circ g$ is close to the identity of $X$.

To see the converse, assume that there exists a uniformly expansive $g: Y\to X$ such that the compositions $f \circ g$ and $g \circ f$ are close to the identity of $Y$ and $X$, respectively. To see that $f$ is coarsely onto, let $y \in Y$. Since $f \circ g$ is close to the identity on $Y$, there exists $r\in(0,1)$ and $t>0$ such that $M_2(y, f(g(y)),t)>1-r$. Said differently, $f(g(y))$ is an element of the image of $f$ that is in the $N_{r,t}$-neighborhood of $y$. In other words, $f$ is coarsely onto. To see that $f$ is effectively proper, let $C>0$ and $t>0$ be arbitrary, and let $x_1,x_2$ be elements of $X$ such that $M_2(f(x_1), f(x_2),t)>1-r$.
By the uniform expansiveness of $g$, there exists $B>0$ and $t'>0$ such that $M_1(g(f(x_1)),g(f(x_2)),t')\geq B$. Since $g \circ f$ is close to the identity on $X$, there exists $r''\in(0,1)$ and $t''>0$ such that $M_1(g(f(x_1)),x_1,t'')>1-r''$ and $M_1(g(f(x_2)),x_2,t'')>1-r''$. Thus, we have
\begin{equation*}
\begin{split}
M_1(x_1,x_2,2t''+t') & \geq M_1(x_1, g(f(x_1)),t'') *_1 M_1(g(f(x_1)), g(f(x_2)),t') *_1 M_1(g(f(x_2)), x_2,t'')\\
& \geq (1-r'') *_1 B *_1(1-r'')
\end{split}
\end{equation*}
Since the right side of the above inequality is positive, there exists $s \in (0,1)$ such that 
\[M_1(x_1,x_2,2t''+t')>1-s,\]
showing that $f$ is effectively proper. This finishes the proof that $f$ is a coarse equivalence.
\end{proof}

\begin{corollary}
Let $(X,M_1,*_1)$ and $(Y, M_2,*_2)$ be fuzzy metric spaces. Let $f:X \to Y$ be a coarse equivalence with a coarse inverse $g: Y\to X$. Then $g$ is also a coarse equivalence with a coarse inverse $f$. \hfill $\square$
\end{corollary}

In light of the above corollary, two fuzzy metric spaces are called \textbf{coarsely equivalent} if there exists a coarse equivalence between them.

\begin{corollary}
Let $(X,d_1)$ and $(Y,d_2)$ be metric spaces and $(X,M_1,*)$ and $(Y, M_2,*)$ be the corresponding standard fuzzy metric spaces. Then $(X,d_1)$ and $(Y,d_2)$ are coarsely equivalent (in the metric sense) if and only if $(X,M_1,*)$ and $(Y, M_2,*)$ are coarsely equivalent (in the fuzzy metric sense).
\end{corollary}
\begin{proof}
By the repeated use of Lemma \ref{useful_lemma}, it is easy to see that $f:X\to Y$ is uniformly expansive in the metric sense if and only if it is uniformly expansive in the fuzzy metric sense. Similarly, by using the same lemma it is easy to see that two maps $f,g:X\to Y$ are close in the metric sense if and only if they are close in the fuzzy metric sense. The conclusion follows then from Proposition \ref{different_way} and the corresponding Lemma 2.2 in \cite{permanence}. For the definitions in the metric sense, the reader is referred to \cite{permanence}.
\end{proof}

In the following theorem, we use the following notation: given a function $f:X\to Y$ and a collection $\mathcal{U}$ of subsets of $X$, by $f(\mathcal{U})$ we mean the collection of images of elements of $\mathcal{U}$ under $f$, i.e., 
\[f(\mathcal{U})=\{f(U) \mid U \in \mathcal{U}\}.\]
Similarly, given a collection $\mathcal{U}$ of subsets of a fuzzy metric space $(X,M,*)$, and given $r\in(0,1)$, $t>0$, by $N_{r,t}(\mathcal{U})$ we mean the collection of $(r,t)$-neighborhoods of elements of $\mathcal{U}$, i.e.,
\[N_{r,t}(\mathcal{U})=\{N_{r,t}(U) \mid U \in \mathcal{U}\}.\]

The reader is encouraged to compare the following Theorem with Proposition 22 in \cite{asymptotic_dimension}.
\begin{theorem}\label{2ndbigtheorem}
Let $(X,M_1,*_1)$ and $(Y, M_2,*_2)$ be fuzzy metric spaces. Let $f:X \to Y$ be a coarse equivalence. Then $\operatorname{asdim}X=\operatorname{asdim}Y.$
\end{theorem}
\begin{proof}
Assume $\operatorname{asdim}X \leq n$. Let $r\in (0,1)$ and $t>0$ be arbitrary. Since $f$ is coarsely onto, there exist $t_1>0$ and $r_1\in(0,1)$ such that for any $y \in Y$, there exists $x \in X$ such that $M_2(f(x),y,t_1)>1-r_1$. Since $\operatorname{asdim}X \leq n$, there exist uniformly bounded $(R,T)$-disjoint families $\mathcal{U}^0, \dots, \mathcal{U}^n$ of subsets of $X$ such that $\bigcup_i \mathcal{U}^i$ is a cover of $X,$ where $R \in (0,1)$ and $T>0$ are yet to be determined. Consider the families $f(\mathcal{U}^0), \dots, f(\mathcal{U}^n).$ To see that these families are uniformly bounded, let $f(x),f(y) \in f(U) \in f(\mathcal{U}^i)$ for some $x,y \in U$ and some integer $i$ such that $0 \leq i \leq n$. Since the families $\mathcal{U}^0, \dots, \mathcal{U}^n$ are uniformly bounded, we know that there are parameters $r'\in(0,1)$ and $t'>0$ such that $M(x_1,x_2,t')>1-r'$ for any $x_1,x_2 \in U \in \mathcal{U}^i$ for any integer $i$ such that $0 \leq i \leq n.$ In particular, $M(x,y,t')>1-r'$.  The uniform expansiveness of $f$ implies that there exists $t''>0$ such that $M_2(f(x),f(y),t'')\geq B$ for some $B>0$. Thus, after letting $r'' \in (0,1)$ be such that $1-r'' < B$ , we see that 
\[M_2(f(x),f(y),t'')\geq B > 1-r''.\]
Since $f(x)$ and $f(y)$ were arbitrary elements of an arbitrary set $f(U)$ belonging to an arbitrary family $f(\mathcal{U}^i)$, this shows that $f(\mathcal{U}^0), \dots, f(\mathcal{U}^n)$ are uniformly bounded.

Consider the families $N_{r_1,t_1}(f(\mathcal{U}^0)), \dots , N_{r_1,t_1}(f(\mathcal{U}^n))$. Since $f$ is coarsely onto with parameters $r_1$ and $t_1$, $N_{r_1,t_1}(f(\mathcal{U}^0)), \dots , N_{r_1,t_1}(f(\mathcal{U}^n))$ cover $Y$. By Proposition \ref{nbhd}, they are also uniformly bounded. It is left to show that the families $N_{r_1,t_1}(f(\mathcal{U}^0)), \dots ,N_{r_1,t_1}(f(\mathcal{U}^n))$ are $(r,t)$-disjoint. Let $x \in N_{r_1,t_1}(f(U)) \in N_{r_1,t_1}(f(\mathcal{U}^i))$ and $y \in N_{r_1,t_1}(f(V)) \in N_{r_1,t_1}(f(\mathcal{U}^i))$ for some integer $i$, where $0 \leq i \leq n$, such that $N_{r_1,t_1}(f(U))\neq N_{r_1,t_1}(f(V))$. In particular, $U \neq V$. Since $x \in N_{r_1,t_1}(f(U))$, there exists $a \in f(U)$ such that $M_2(a,x,t_1)>1-r_1$. Similarly, since  $y \in N_{r_1,t_1}(f(V))$, there exists $b \in f(V)$ such that $M_2(b,y,t_1)>1-r_1$. Thus, we have:
\begin{equation*}\label{eq1}
\begin{split}
M_2(a,b,2t_1+t) & \geq M_2(a,x, t_1) *_2M_2(x,y, t) *_2M_2(y,b, t_1) \\
& \geq (1-r_1) *_2M_2(x,y, t) *_2(1-r_1)\\
\end{split}
\end{equation*}
Since $a \in f(U),$ there exists $a' \in U$ such that $f(a')=a$. Similarly, since $b \in f(V)$, there exists $b' \in V$ such that $f(b')=b.$ Notice that by the effective properness, the left side of the above inequality can be chosen to be as small as one likes. This is because the contrapositive of the implication in the definition of effective properness says that for all $C>0$ and $2t_1+t>0$, there exists $D>0$ and $t_2>0$ such that $M_1(a',b',t_2)<D$ implies $M_2(a,b,2t_1+t)<C$, and we can make $M_1(a',b',t_2)$ arbitrarily small by the $(R, T)$-disjointness of the family $\mathcal{U}^i,$ since $R\in (0,1)$ and $T>0$ are still to be chosen. Since $(1-r_1) *_2 s *_2(1-r_1)$ is a non-decreasing continuous function of $s \in [0,1]$ such that
\begin{itemize}
\item $(1-r_1) *_2 0 *_2(1-r_1)=0$,
\item $(1-r_1) *_2 1 *_2(1-r_1)=(1-r_1)*_2(1-r_1)>0,$
\end{itemize}
there exists $\epsilon \in (0,1)$ such that if $(1-r_1) *_2 s *_2(1-r_1)\leq \epsilon$, then $s<1-t-l$ for some small $l>0$ such that $1-t-l \in (0,1)$. Thus, let us choose $R \in (0,1)$ and $T>0$ so that $M_2(f(u_1),f(u_2),2t_1+t) \leq \epsilon$ for any $u_1, u_2$ belonging to distinct sets from the same family $\mathcal{U}^i$ for any $i$. In particular, $M_2(a,b,2t_1+t) \leq \epsilon$. This gives
\begin{equation*}
\begin{split}
\epsilon & \geq M_2(a,b,2t_1+t)\\
& \geq M_2(a,x, t_1) *_2M_2(x,y, t) *_2M_2(y,b, t_1) \\
& \geq (1-r_1) *_2M_2(x,y, t) *_2(1-r_1),\\
\end{split}
\end{equation*}
which shows that $M_2(x,y, t)<1-t-l<1-t$. Since $x$ was an arbitrary element of $N_{r_1,t_1}(f(U))$, $y$ was an arbitrary element of $N_{r_1,t_1}(f(V))$, $N_{r_1,t_1}(f(U))\neq N_{r_1,t_1}(f(V))$, and both $N_{r_1,t_1}(f(U))$ and $N_{r_1,t_1}(f(V))$ were arbitrary elements of an arbitrary family $N_{r_1,t_1}(f(\mathcal{U}^i))$ for some integer $i$ where $0 \leq i \leq n$, this shows that the the families 
\[N_{r_1,t_1}(f(\mathcal{U}^0)), \dots , N_{r_1,t_1}(f(\mathcal{U}^n))\]
are $(r,t)$-disjoint, finishing the proof that 
\[\operatorname{asdim}X \leq n \implies \operatorname{asdim}Y \leq n.\]
The same proof applied to the coarse inverse of $f$ implies that 
\[\operatorname{asdim}Y \leq n \implies \operatorname{asdim}X \leq n. \qedhere\]
\end{proof}

\begin{corollary}\label{importantcorollary}
Let $(X,M_1,*_1)$ and $(Y, M_2,*_2)$ be fuzzy metric spaces. Let $f:X \to Y$ be a coarse embedding. Then $\operatorname{asdim}X \leq \operatorname{asdim}Y.$ \hfill $\square$
\end{corollary}
\iffalse
\begin{proof}
Since $f$ is a coarse equivalence onto $f(X)$ (equipped with the subspace fuzzy metric structure), Theorem \ref{2ndbigtheorem} implies that $\operatorname{asdim}X = \operatorname{asdim} f(X)$. Since $f(X) \subseteq Y$, by Proposition \ref{subspace} we know that $\operatorname{asdim} f(X) \leq \operatorname{asdim} Y$. Thus,
\[\operatorname{asdim}X=\operatorname{asdim}f(X) \leq \operatorname{asdim} Y. \qedhere\]
\end{proof}
\fi

\section{Examples}\label{examples}
In this section, we explicitly calculate asymptotic dimension of several fuzzy metric spaces.

\begin{example}\label{example0}
Let $(X,M,*)$ be any bounded fuzzy metric space. Then $\operatorname{asdim}X=0$.
\end{example}
\iffalse
\begin{proof}
Setting $\mathcal{U}^0=\{X\}$ gives a uniformly bounded family that is vacuously $(r,t)$-disjoint for any $r \in (0,1)$ and $t>0$.
\end{proof}
\fi

\begin{example}\label{example1}

One can apply Proposition 3.9 to the metric spaces with known asymptotic dimension. In particular,  $\operatorname{asdim} \mathbb{R}^n=\operatorname{asdim} \mathbb{Z}^n=n$, where $\mathbb{R}^n$ and $\mathbb{Z}^n$ are endowed with the standard fuzzy metric induced by the Euclidean metric.
%Let $X=\mathbb{R}^n$ or $X=\mathbb{Z}^n$ for any positive integer $n$. Equip $X$ with the standard fuzzy metric structure induced by the Euclidean metric on $\mathbb{R}^n$. Then $\operatorname{asdim}X=n$.
\end{example}
\iffalse
\begin{proof}
%The set $\mathbb{R}^n$ equipped with the Euclidean metric has asymptotic dimension (in the metric sense) equal to $n$ (see Proposition 2.2.7 in \cite{bedlewo}). By Proposition $20$ in \cite{asymptotic_dimension}, $\mathbb{R}^n$ is coarsely equivalent (in the metric sense) to $\mathbb{Z}^n$ when $\mathbb{Z}^n$ is equipped with the subspace metric. Since coarsely equivalent metric spaces have the same asymptotic dimension (Proposition 22 in \cite{asymptotic_dimension}), asymptotic dimension of $\mathbb{Z}^n$ is also $n$.  

Since $\operatorname{asdim}\mathbb{R}^n= \operatorname{asdim}\mathbb{N}^n = n$ in the Euclidean metric sense, Proposition \ref{standardandfuzzy} implies that the standard fuzzy metric induced by the Euclidean metric on $\mathbb{R}^n$ and $\mathbb{Z}^n$ also has asymptotic dimension $n$.
\end{proof}
\fi

\begin{example}\label{example2}
Let $X=\mathbb{N}$. Let $*$ be a continuous $t$-norm defined by $a*b=ab$ for all $a,b \in [0,1]$. Define
\[
M(x,y,t)=
\begin{cases}
1 & \text{if } x=y,\\
\frac{1}{xy} & \text{otherwise},
\end{cases}
\] 
for all $x,y \in X$ and all $t>0$. Then $(X,M,*)$ is a fuzzy metric space such that $\operatorname{asdim}X=0$.
\end{example}
\begin{proof}
It is straightforward to verify that $(X,M,*)$ is a fuzzy metric space. In fact, $(X,M,*)$ is a fuzzy metric space that is not a standard fuzzy metric space for any metric (see Example 2.10 in \cite{uniform}). Let $r \in (0,1)$ and $t>0$ be arbitrary. Find $N$ large enough such that
\[\frac{1}{N+1}< 1-r.\]
Let $U_0=\{n \in X \mid n\leq N\}$. For any $n \in \mathbb{N}$, define $U_n=\{N+n\}$. Clearly the family $\mathcal{U}^0 = \{U_n \mid n \in \{0\}\cup \mathbb{N}\}$ is a cover of $X$. Since $U_0$ is finite and $U_n$ is a singleton set for any $n \in \mathbb{N}$, the cover $\mathcal{U}^0$ in uniformly bounded. To see that the family $\mathcal{U}^0$ is $(r,t)$-disjoint, notice that for any $m,m'\geq N$ such that $m \neq m'$, we have that 
\[\frac{1}{mm'} \leq \frac{1}{N(N+1)}\leq\frac{1}{N+1}<1-r.\]
Thus, all the singleton sets in $\mathcal{U}^0$ are $(r,t)$-disjoint. Also, for any $n \in U_0$ and any $m\geq N+1$, we have
\[\frac{1}{nm}\leq \frac{1}{N+1}<1-r,\]
showing that $U_0$ is $(r,t)$-disjoint from all the singleton sets in $\mathcal{U}^0$. This finishes the proof that the family $\mathcal{U}^0$ is $(r,t)$-disjoint. Since $r \in (0,1)$ and $t>0$ were arbitrary, this shows that $\operatorname{asdim}X=0$.
\end{proof}

Notice that there exists no effectively proper map from the fuzzy metric space in Example \ref{example2} to the fuzzy metric space in Example \ref{example0}. This means that the fuzzy metric space from Example \ref{example2} and the fuzzy metric space from Example \ref{example0} are examples of fuzzy metric spaces that are not coarsely equivalent, but have the same asymptotic dimension.

\begin{example}\label{example3}
Let $X=\mathbb{N}$. Let $*$ be a continuous $t$-norm defined by $a*b=ab$ for all $a,b \in [0,1]$. Define
\[
M(x,y,t)=
\begin{cases}
\frac{x}{y} & \text{if } x\leq y,\\
\frac{y}{x} & \text{if } y \leq x,
\end{cases}
\] 
for all $x,y \in X$ and all $t>0$. Then $(X,M,*)$ is a fuzzy metric space such that $\operatorname{asdim}X=1$.
\end{example}
\begin{proof}
It is easy to see that $(X,M, *)$ is a fuzzy metric space that is not a standard fuzzy metric space for any metric (see Example 2.11 and Remark 2.12 in \cite{main_paper}). To see that $\operatorname{asdim}X\neq 0$, first notice that any bounded set in $X$ must be finite. However, for any $r \in (0,1)$ and $t>0$, we can find an integer $N$ large enough such that for any integer $m \geq 0$, we have that $M(m,m+1,t)=\frac{m}{m+1}>1-r$. Thus, an $(r,t)$-disjoint family of subsets of $X$ must include the set containing all the integers bigger than $N$. Thus, this family cannot be uniformly bounded.

To see that $\operatorname{asdim}X =1$, let $r \in (0,1)$ and $t>0$ be arbitrary. Construct the family $\mathcal{U}$ of subsets of $X$ in the following way:
\begin{align*}
U_0 &= \{1\},\\
U_1 &= \{a_1+i \mid 0 \leq i \leq m_1 \},
\intertext{where $a_1$ is the smallest integer greater than 1 such that $\frac{1}{a_1}<1-r$ and $m_1$ is the smallest non-negative integer such that $\frac{a_1-1}{a_1+m_1+1}<1-r$, and for all $n\geq 2$, define}
U_n &= \{a_n+i \mid 0 \leq i \leq m_n \}
\end{align*}
where $a_n$ is the smallest integer greater than $a_{n-1}+m_{n-1}$ such that $\frac{a_{n-1}+m_{n-1}}{a_n}<1-r$ and $m_n$ is the smallest non-negative integer such that $\frac{a_n-1}{a_n+m_n+1}<1-r$. Let $\mathcal{U}=\{U_n \mid n \geq 0\}$. To create a family $\mathcal{V}$, let $V_0$ consist of all the integers between $U_0$ and $U_1$, $V_1$ consist of all the integers between $U_1$ and $U_2,$ and in general for any $n\geq 2$,
\[V_n=\{n \in X \mid a_n+m_n<n<a_{n+1}\}.\]
Set $\mathcal{V}=\{V_n \mid n \geq 0\}$. Clearly the families $\mathcal{U}$ and $\mathcal{V}$ form a cover of $X$. To finish the proof, we will show that they are also $(r,t)$-disjoint and uniformly bounded.\\ 
\textbf{Claim 1}: The families $\mathcal{U}$ and $\mathcal{V}$ are $(r,t)$-disjoint.\\
To show $(r,t)$-disjointness of families $\mathcal{U}$ and $\mathcal{V}$, we need to show that $M(x,y,t)<1-r$ for any two $x$ and $y$ belonging to different sets in the same family. Notice that for any positive integers $a$ and $b$ such that $a<b$, we have that $M(a,b,t)=\frac{a}{b}\geq \frac{c}{d}=M(c,d,t)$ for any $c\leq a$ and any $d \geq b$. Thus, to check the $(r,t)$-disjointness of the families $\mathcal{U}$ and $\mathcal{V}$, it is enough to show that the largest element of the set $U_n$ ($V_n$, respectively) is $(r,t)$-disjoint from the smallest element of the set $U_{n+1}$ ($V_{n+1}$, respectively) for any non-negative integer $n$. To see that the family $\mathcal{U}$ is $(r,t)$-disjoint, notice that the definition of $a_1$ ensures that the smallest element of $U_1$ (namely $a_1$) is $(r,t)$-disjoint from the largest element of the set $U_{0}$ (namely $1$), since 
\[M(1, a_1,t)=\frac{1}{a_1}<1-r.\]
Similarly, for all $n> 1$ the definition of $a_n$ ensures that the smallest element of $U_n$ (namely $a_n$) is $(r,t)$-disjoint from the largest element of the set $U_{n-1}$ (namely $a_{n-1}+m_{n-1}$), since 
\[M(a_{n-1}+m_{n-1}, a_n,t)=\frac{a_{n-1}+m_{n-1}}{a_n}<1-r.\]
This shows that the family $\mathcal{U}$ is $(r,t)$-disjoint. To see that the family $\mathcal{V}$ is $(r,t)$-disjoint, first assume that all the elements in $\mathcal{V}$ are non-empty, as if some of them are empty, then they satisfy the $(r,t)$-disjointness condition vacuously. Notice that the definition of $m_1$ ensures that   the smallest element of $V_1$ (namely $a_1+m_1+1$) is $(r,t)$-disjoint from the largest element of the set $V_{0}$ (namely $a_1-1$), since
\[M(a_1+m_1+1, a_1-1,t)=\frac{a_1-1}{a_1+m_1+1}<1-r.\]
Similarly, for all $n > 1$ the definition of $m_n$ ensures that the smallest element of $V_n$ (namely $a_n+m_n+1$) is $(r,t)$-disjoint from the largest element of the set $V_{n-1}$ (namely $a_n-1$), since
\[M(a_n+m_n+1, a_n-1,t)=\frac{a_n-1}{a_n+m_n+1}<1-r.\]
This shows that the family $\mathcal{V}$ is also $(r,t)$-disjoint.\\
\textbf{Claim 2}: The families $\mathcal{U}$ and $\mathcal{V}$ are uniformly bounded.\\
We will show that the families $\mathcal{U}$ and $\mathcal{V}$ are uniformly bounded with parameters $r$ and $t$ chosen at the beginning of this proof. To see that the families $\mathcal{U}$ and $\mathcal{V}$ are uniformly bounded, we need to show that $M(x,y,t)>1-r$ for any two $x$ and $y$ belonging to the same set in any of the families. By using the same observation as in Claim 1, it is enough to show that this inequality is true for the largest element of the set $U_n$ ($V_n$, respectively) and the smallest element of the set $U_{n}$ ($V_{n}$, respectively) for any non-negative integer $n$. To see that the family $\mathcal{U}$ is uniformly bounded, recall that for any $n\geq 1$, the smallest element of $U_n$ is $a_n$ and the largest element of $U_n$ is $a_n+m_n$. Also, recall that for any $n\geq 1$, the definition of $m_n$ says that $m_n$ is the smallest non-negative integer such that $\frac{a_{n}-1}{a_n+m_n+1}<1-r$. This means that $\frac{a_{n}-1}{a_n+m_n}\geq 1-r$, implying that
\[M(a_n, a_n+m_n,t)=\frac{a_n}{a_n+m_n}> \frac{a_n-1}{a_n+m_n}\geq 1-r\]
for all $n \geq 1$. Since the set $U_0$ is a singleton set, this shows that the family $\mathcal{U}$ is uniformly bounded.

To see that the family $\mathcal{V}$ is uniformly bounded, recall that for any $n\geq 1$, the smallest element of $V_n$ is $a_n+m_n+1$ and the largest element of $V_n$ is $a_{n+1}-1$. Also, recall that for any $n\geq 2$, the definition of $a_n$ says that $a_n$ is the smallest integer greater than $a_{n-1}+m_{n-1}$ such that $\frac{a_{n-1}+m_{n-1}}{a_n}<1-r$. This means that $\frac{a_{n}+m_{n}}{a_{n+1}}\geq 1-r$ for all $n \geq 1$, implying that 
\[M(a_n+m_n+1,a_{n+1}-1,t)=\frac{a_n+m_n+1}{a_{n+1}-1}>\frac{a_n+m_n}{a_{n+1}-1}\geq 1-r\]
for all $n \geq 1$. To see that $V_0$ is bounded with the same parameters, notice that by the definition of $a_1$ we must have that $\frac{1}{a_1-1}\geq 1-r$, implying that 
\[M(2,a_1-1,t)=\frac{2}{a_1-1} > \frac{1}{a_1-1}\geq 1-r,\]
finishing the proof that the family $\mathcal{U}$ is also uniformly bounded.
\end{proof}

\iffalse
Since the asymptotic dimension of the fuzzy metric space from Example \ref{example2} differs from the asymptotic dimension of the fuzzy metric space from Example \ref{example3}, Theorem \ref{2ndbigtheorem} implies that these two fuzzy metric spaces are not coarsely equivalent. Also, Corollary \ref{importantcorollary} implies that the fuzzy metric space from Example \ref{example3} does not coarsely embed in the fuzzy metric space from Example \ref{example2}.
\fi

Recall that a fuzzy metric space $(X, M, *)$ is called \textbf{non-Archimedean} if
\[M(x,y,t)*M(y,z,t') \leq M(x,z, \max\{t,t'\})\]
for any $x,y,z \in X$, and any $t,t'>0$, or equivalently,
\[M(x,y,t)*M(y,z,t) \leq M(x,z,t) \]
for any $x,y,z \in X$, and any $t>0$ (see for example \cite{rzeszow}).
The reader is encouraged to compare the following example with Theorem 3 in \citep{rzeszow}.
\begin{example}\label{example4}
Let $(X, M, *)$ be any non-Archimedean fuzzy metric space with a continuous $t$-norm defined by  $a*b=\min\{a,b\}$ for all $a,b \in [0,1]$. Then $\operatorname{asdim}X=0$.
\end{example}
\begin{proof}
Let $r \in (0,1)$ and $t>0$ be arbitrary. Let $\mathcal{U}=\{B(x,r+\varepsilon,t) \mid x \in X\}$ for some small $\varepsilon>0$ such that $r+\varepsilon \in (0,1)$. Clearly $\mathcal{U}$ is uniformly bounded. To see that $\mathcal{U}$ is $(r,t)$-disjoint, let $B(a,r+\varepsilon,t)$ and $B(b,r+\varepsilon,t)$ be arbitrary elements of $\mathcal{U}$ for some $a,b \in X$ such that $a\neq b$. We will show that if $B(a,r+\varepsilon,t)$ and $B(b,r+\varepsilon,t)$ are not $(r,t)$-disjoint, then they are equal. Assume that $B(a,r+\varepsilon,t)$ and $B(b,r+\varepsilon,t)$ are not $(r,t)$-disjoint. This means that there exist $a' \in B(a,r+\varepsilon,t)$ and $b' \in B(b,r+\varepsilon,t)$ such that $M(a',b',t)>1-r-\varepsilon$. But then for an arbitrary point $b'' \in  B(b,r,t)$, we have that
\begin{equation*}
\begin{split}
M(a,b'',t) & \geq M(a,a',t)*M(a',b',t)*M(b',b,t)*M(b,b'',t)\\
& =\min\left\{M(a,a',t),M(a',b',t),M(b',b,t),M(b,b'',t)\right\}\\
& > \min\left\{(1-r-\varepsilon),(1-r-\varepsilon),(1-r-\varepsilon), (1-r-\varepsilon)\right\} \\
& = 1-r-\varepsilon
\end{split}
\end{equation*}
This shows that $b'' \in B(a,r+\varepsilon,t)$. Since $b''$ was an arbitrary element of $B(b,r+\varepsilon,t)$, this shows that $B(a,r+\varepsilon,t) \subseteq B(b,r+\varepsilon,t)$. By symmetry, $B(b,r+\varepsilon,t) \subseteq B(a,r+\varepsilon,t)$.
\end{proof}

\iffalse Since asymptotic dimension of a fuzzy metric space does not depend on the continuous $t$-norm on that fuzzy metric space, the above example shows that the fuzzy metric space in Example \ref{example3} is not a fuzzy metric space if we replace its continuous $t$-norm with the minimum $t$-norm from example \ref{example4}.
\fi

%\cite{reza}
%\bibliographystyle{abbrv}
%\bibliography{fuzzymetricspaces}{}

\end{document}